\documentclass[12pt,a4paper]{article}

\usepackage{amsmath,amsthm,amssymb,amscd,a4wide}

\usepackage{dsfont}
\usepackage{mathrsfs}
\usepackage{setspace}
\usepackage{hyperref}

\usepackage{xcolor}

\newcommand{\ud}{\mathrm{d}}

\newcommand{\cD}{{\mathcal D}}

\newcommand{\eins}{\mathds{1}}

\numberwithin{equation}{section}

\newtheorem{theorem}{Theorem}[section]
\newtheorem{lemma}[theorem]{Lemma}
\newtheorem{prop}[theorem]{Proposition}

\theoremstyle{definition}

\newtheorem{rem}[theorem]{Remark}

\numberwithin{equation}{section}

\begin{document}

\thispagestyle{empty}

\vspace*{1cm}

\begin{center}

{\LARGE\bf On the spectral gap of one-dimensional Schrödinger operators on large intervals} \\

\vspace*{2cm}

{\large Joachim Kerner \footnote{E-mail address: {\tt Joachim.Kerner@fernuni-hagen.de}} and Matthias Täufer \footnote{E-mail address: {\tt Matthias.Taeufer@fernuni-hagen.de}} }%

\vspace*{5mm}

Department of Mathematics and Computer Science\\
FernUniversit\"{a}t in Hagen\\
58084 Hagen\\
Germany\\

\end{center}

\vfill

\begin{abstract}  We study the effect of non-negative potentials on the spectral gap of one-dimensional Schrödinger operators in the limit of large intervals. In particular, we derive upper and lower bounds on the gap for different classes of potentials which characterize its asymptotic behaviour.
\end{abstract}

\newpage

\section{Introduction}
The \textit{spectral gap} is a classical quantity in spectral theory of one-dimensional Schr\"odinger operators and is defined as the difference between the lowest and the second lowest eigenvalue. 
In this note, we study its asymptotic behaviour as the length $L > 0$ of the underlying interval tends to infinity. 
There is a long history of articles where the dependence of the gap on a potential on a fixed interval has been studied, such as  \cite{AB,Abramovich,Lavine,ACH} and, with different methods, \cite{KirschGapII,KirschGap}. 
It turns out that the effect of a potential on the spectral gap strongly depends on its shape: If the potential is convex or a symmetric single-well potential, the spectral gap increases in comparison to the spectral gap of the free Dirichlet Laplacian \cite{AB,Lavine}. Contrarily, it was shown in \cite{Abramovich} that adding a twice symmetric double-well potential to the free Dirichlet Laplacian will decrease the spectral gap.
Therefore, it seems in general very hard to control the spectral gap on a fixed interval for a generic potential.
In this note, we study instead the asymptotic behaviour of the gap for a fixed potential as the length of the underlying interval tends to infinity.
It is one aim of this note to show that in this scenario the asymptotic behaviour of the spectral gap can nevertheless be controlled for a large class of potentials without symmetry or related assumptions. 
As a matter of fact, we will restrict ourselves to bounded and non-negative potentials on $\mathbb{R}$ which decay sufficiently fast at infinity. In particular, our results cover arbitrary compactly supported bounded potentials. 
We show that the gap decays at least as fast as the one of the free Dirichlet Laplacian, that is proportional to $L^{-2}$ as $L$ approaches infinity, which is shown to be sharp. However, we also prove that in many cases the gap closes faster.
Indeed, Theorem~\ref{ThmLowerBoundI} formulates a somewhat striking result that (arbitrarily small and locally supported or fast decaying) potentials will completely change the dynamics of the spectral gap on large intervals from $L^{-2}$ to a strictly faster decay rate.
We conjecture $L^{-3}$ to be the universal decay rate in this class.

Another motivation to study the asymptotic behaviour of the spectral gap grew out of the investigations in \cite{KPSBEC}. In this paper, the authors proved Bose--Einstein condensation of a certain type in external random potentials under the assumption that the underlying one-particle operator (which simply is a $d$-dimensional Schrödinger operator) satisfies a certain \emph{spectral gap condition}. This gap condition is directly related to a lower bound on the spectral gap of this one-particle operator in the limit of large intervals. More explicitly and translated to our deterministic setting, the authors ask to establish a lower bound on the spectral gap which is of order $L^{-1+\eta}$ where $0 < \eta < 1$. In other words, they require the spectral gap not to close too fast as $L$ goes to infinity! While we show in this note that, for potentials of short range, the gap closes strictly faster than required for this gap condition, we nevertheless see it as a first contribution to a better understanding of the asymptotics of the spectral gap of Schrödinger operators.

\section{Model and results}
We study one-dimensional Schrödinger operators of the form
\begin{equation*}
h=-\frac{\ud^2}{\ud x^2}+v(x)
\end{equation*}
with a bounded non-negative potential $v:\mathbb{R} \rightarrow \mathbb{R}_{\geq 0}$. The restriction of $h$ to the finite interval $(-L/2,+L/2)$ is denoted by $h^D_L$ where $D$ refers to Dirichlet boundary conditions imposed at $x=\pm L/2$. It is well-known that $h^D_L$ has purely discrete spectrum and we denote its eigenvalues with multiplicities as $\varepsilon_0(L) < \varepsilon_1(L) \leq ...$. An associated orthonormal basis of eigenfunctions shall be denoted by $(\varphi^L_n)_{n \in \mathbb{N}_0}$.

The object of interest in this paper is the \textit{spectral gap} defined by 
\begin{equation*}
\Gamma_v(L):=\varepsilon_1(L)-\varepsilon_0(L)\ .
\end{equation*}
Note that, since $v\in L^{\infty}((-L/2,+L/2))$, the ground state $\varphi^L_0$ is unique \cite{lieb2001analysis} and hence $\Gamma_v(L) > 0$ for any fixed $L > 0$. To warm up, we start with an upper bound for potentials of sufficiently short range.
\begin{theorem}[Upper bound I]\label{UpperBoundI} Let $v \in L^{\infty}(\mathbb{R})$ be a non-negative potential such that
	\begin{equation*}\label{ShortRange}
	|v(x)|\leq \frac{C}{|x|^2}\ , \quad \text{for a.e.}\quad x\in \mathbb{R}\ ,
	\end{equation*}
	for some constant $C > 0$. Then
	\begin{equation*}
	\Gamma_v(L) \leq \frac{\beta}{L^2}
	\end{equation*}
	for some constant $\beta$ and all $L$ large enough.
\end{theorem}
\begin{proof} Obviously, one has $\Gamma_v(L) \leq \varepsilon_1(L)$, and it remains to find an upper bound on $ \varepsilon_1(L)$. 
For this purpose, consider the restriction of $h$ to the interval $(L/4,L/2)$ with Dirichlet boundary conditions at the end points of the interval. We denote this operator by $\tilde{h}^{D}_L$ and its second lowest eigenvalue by $\tilde{\varepsilon}_1(L)$. 
Due to the min-max principle, one has $\varepsilon_1(L) \leq \tilde{\varepsilon}_1(L)$.
	Now, since $\|v\|_{L^{\infty}(L/4,L/2)} \leq 16C/L^2$, we obtain
	\begin{equation*}
	\tilde{\varepsilon}_1(L) \leq \frac{64\pi^2}{L^2}+\frac{16C}{L^2}\ 
	\end{equation*}
	where we used that the second Dirichlet eigenvalue on $(L/4,L/2)$ is $\frac{64\pi^2}{L^2}$.
\end{proof}
\begin{rem} Clearly, the proof of Theorem~\ref{UpperBoundI} would also work for potentials that decay like $|x|^{-2}$ on either the negative or the positive half-axis only. 
In fact, in this case, one can replace uniform boundedness of $v$ by local boundedness, which includes -- for instance --  the potential $v(x) = \exp (x)$.
\end{rem}
By adapting the methods developed in \cite{Lavine,ACH} we can also establish the following result which yields a similar upper bound as in Theorem~\ref{UpperBoundI} for potentials of arbitrarily slow decay. 
However, it requires a symmetry and monotonicity condition.
\begin{theorem}[Upper bound II]\label{UpperBoundII} Let $v \in L^{\infty}(\mathbb{R})$ be non-negative and such that $v(-x)=v(x)$. In addition, $v$ shall be monotonically increasing on $(-\infty,0)$ and thus monotonically decreasing on $(0,\infty)$. Then,
	\begin{equation*}
	\Gamma_v(L) \leq \frac{3\pi^2}{L^2}
	\end{equation*}
	for all $L > 0$.
\end{theorem}
\begin{proof} One introduces the function
	\begin{equation*}
	\psi^L(x):=|\varphi^L_1(x)|^2-|\varphi^L_0(x)|^2
	\end{equation*}
	and observes that $\psi^L$ is continuous, symmetric and has mean value zero. Furthermore, by Sturm's oscillation theorem, $\varphi^L_1$ has exactly one zero and due to symmetry, we conclude that $\varphi^L_1(0)=0$. Now, as in [\cite{ACH}, Lemma~2.4], one shows that $\psi^L$ has at least one but at most two zeroes in $(-L/2,+L/2)$. Moreover, there exists an (possible empty) interval $(-x_0,+x_0) \subset (-L/2,+L/2)$ such that $\psi^L(x) \leq 0$ for $x \in (-x_0,+x_0)$ and that $\psi^L(x) \geq 0$ for $x \in (-L/2,+L/2) \setminus (-x_0,+x_0)$.
	
	One then employs the Hellmann-Feynman formula (see, e.g., \cite{ACH}) which gives an expression for the ``time''-derivative of the gap $\Gamma_v(L)$ given we set $v_t(x):=t\cdot v(x)$. More explicitly, one has
	\begin{equation}\label{FeynHell}
	\frac{\ud}{\ud t}\Gamma_{v_t}(L)=\int_{-L/2}^{+L/2}v(x)\psi^L(x)\ \ud x\ .
	\end{equation}
	Now the key observation is that, since $\psi^L$ has mean value zero, one can replace $v(x)$ in \eqref{FeynHell} by $v(x)-c$ where $c \in \mathbb{R}$ is some constant. More explicitly, we set $c:=v(x_0)$. This yields
	\begin{equation}\label{HF}\begin{split}
	\frac{\ud}{\ud t}\Gamma_{v_t}(L)&=\int_{-L/2}^{+L/2}(v(x)-v(x_0))\psi^L(x)\ \ud x \\
	&=\int_{(-L/2,+L/2) \setminus (-x_0,+x_0)}(v(x)-v(x_0))\psi^L(x)\ \ud x \\ & \quad +\int_{(-x_0,+x_0)}(v(x)-v(x_0))\psi^L(x)\ \ud x \\ 
	&\leq 0\ .
	\end{split}
	\end{equation}
	Integrating this inequality then yields the statement.
	\end{proof}
In the next step we investigate how fast the gap can actually close and establish lower bounds on the spectral gap. For this purpose, the following unitary transformation will come handy: We introduce an operator
\begin{equation*}
g_L=-\frac{\ud^2}{\ud x^2}+w_L(x)
\end{equation*}
on the Hilbert space $L^2((-1/2,+1/2))$ subject to Dirichlet boundary conditions where the (effective) potential $w_L$ is defined via $w_L(x):=L^2v(Lx)$. 
\begin{prop}[Unitary transformation]\label{UnitaryTrafo} The operator $h^D_L$ defined on $L^2((-L/2,+L/2))$ is unitarily equivalent to the operator $L^{-2}g_L$ defined on $L^2((-1/2,+1/2))$.
\end{prop}
\begin{proof} The proof is a straightforward calculation with the unitary map 
	\begin{equation*}
	U: L^2((-L/2,+L/2)) \rightarrow L^2((-1/2,+1/2))\ , \quad (U\varphi)(x):=\sqrt{L}\varphi(Lx)\ .
	\qedhere
	\end{equation*}
\end{proof}
Again, since $w_L \in L^{\infty}((-1/2,+1/2))$, the operator $g_L$ has purely discrete spectrum. We denote its eigenvalues by $\lambda_0(L) < \lambda_1(L) \leq ...$ and its associated spectral gap by 
\begin{equation}\label{GapTH}
\widetilde{\Gamma}_{w_L}(L):=\lambda_1(L)-\lambda_0(L)\ .
\end{equation}
We then obtain the following statement which shows that the spectral gap of $h^D_L$ closes strictly faster than $L^{-2}$, in contrast to the free Dirichlet Laplacian. 
\begin{theorem}[Lower bound]\label{ThmLowerBoundI} Assume that $v \in L^{\infty}(\mathbb{R})$ is non-negative and not the zero potential. In addition, we assume it decays such that
	\begin{equation*}
	|v(x)| \leq \frac{C}{|x|^{\alpha}}\ , \quad \text{for a.e.}\quad x\in \mathbb{R}\ ,
	\end{equation*}
	for some constants $C > 0$ and $\alpha > 2$. Then 
	\begin{equation}\label{EqXXX}
 \lim_{L \rightarrow \infty}L^2\Gamma_v(L)=0\ .
	\end{equation}
\end{theorem}
\begin{rem} Theorem~\ref{ThmLowerBoundI} shows something striking: We know that the gap closes like $L^{-2}$ for the free Dirichlet Laplacian. However, as soon as we add a non-zero and non-negative potential of compact support and arbitrarily small $L^{\infty}$-norm, the gap closes faster than in the free case! In this sense, although one eventually ends up on a very large interval, one nevertheless ``feels'' the potential. Consequently, going from a zero or constant potential to a non-constant potential of compact support can be regarded as a sharp ``phase transition''.
	
\end{rem}

\begin{proof} We first prove the statement for compactly supported potentials and outline the generalization at the end of the proof. 
	
	By Prop.~\ref{UnitaryTrafo} it is enough to prove that $$\lim_{L \rightarrow \infty}\widetilde{\Gamma}_{w_L}(L)=0 \ .$$ 
	
In a first step we show that 
\begin{equation}\label{ToPrive}
\limsup_{L \rightarrow \infty}\lambda_1(L)\leq \frac{\pi^2}{(1/2)^2}\ .
\end{equation}
By the min-max principle we have
\begin{equation*}
\lambda_1(L) \leq \sup_{\phi \in \text{span}\{\psi_1,\psi_2\}, \|\phi\|_{L^2((-1/2,+1/2))}=1}\langle \phi, g_L \phi \rangle_{L^2((-1/2,+1/2))}\ 
\end{equation*}
for any $\psi_1, \psi_2 \in H^1_0((-1/2,+1/2))$, the form domain corresponding to $g_L$, where the expression $\langle \phi, g_L \phi \rangle_{L^2((-1/2,+1/2))}$ is understood as the corresponding quadratic form. We take $\psi_1$ to be the (normalized) ground state of the Dirichlet Laplacian on the interval $(-1/2,-\varepsilon)$ and $\psi_2$ to be the (normalized) ground state of the Dirichlet Laplacian on the interval $(\varepsilon,1/2)$, both extended by zero to all of $(-1/2,+1/2)$. For $L > 0$ large enough, $w_L$ has support only in the interval $[-\varepsilon,+\varepsilon]$ and we readily obtain
\begin{equation}\label{EstXX}
\lambda_1(L) \leq \frac{\pi^2}{(1/2-\varepsilon)^2}
\end{equation}
for $L$ large enough. Since $\varepsilon$ was arbitrary, this proves \eqref{ToPrive}.

Next, we show
\begin{equation}\label{ToPriveI}
  \liminf_{L \rightarrow \infty}\lambda_0(L)\geq \frac{\pi^2}{(1/2)^2}\ .
\end{equation}
There is a relatively short proof of~\eqref{ToPriveI}, see Lemma~\ref{lem:appendix} in the appendix.
In contrast to it, the argument we are going to use here is longer but, as a by-product, it yields information on the convergence of the ground states (at least along a subsequence), which will then be crucial in the proof of the subsequent Theorem~\ref{TheoremSP}. 

First observe that $\lambda_0(L)\leq C$ for all $L$ and some constant $C > 0$. 
This immediately follows from the variational principle. 

Now, to prove \eqref{ToPriveI}, we assume for contradiction that there exists a sequence of lengths $(L_k)$, tending to $\infty$ with $\lambda_0(L_k) \rightarrow c < \frac{\pi^2}{(1/2)^2}$ as $k \rightarrow \infty$. Let $(\varphi_k)$ denote the corresponding sequence of normalized ground states of the operator $g_{L_k}$. We claim that there exists a subsequence of $(\varphi_k)$ that converges in $H^1$ to a limit $\varphi_{\infty} \in H^1((-1/2,+1/2))$ which is such that $\varphi_{\infty}(0)=0$. Assuming this for the moment, we conclude, for an arbitrary $\varepsilon > 0$ and $j$ large enough,
\begin{equation*}\begin{split}
\lambda_0(L_{k_j})=\langle \varphi_{k_j},g_{L_{k_j}}\varphi_{k_j}\rangle_{L^2((-1/2,+1/2))} &\geq \|\varphi^{\prime}_{k_j}\|^2_{L^2((-1/2,+1/2))}\\
&\geq \|\varphi^{\prime}_{\infty}\|^2_{L^2((-1/2,+1/2))}-\varepsilon \\
&\geq \frac{\pi^2}{(1/2)^2}-\varepsilon\ ,
\end{split}
\end{equation*}
where we employed the Poincar\'e inequality with optimal constant in the last step, taking $\varphi_{\infty}(0)=0$ into account. However, since $\varepsilon > 0$ was arbitrary, we end up with a contradiction.

It remains to prove existence of $\varphi_{\infty}$ with the desired properties. To do this, we first note that for $x \in (-1/2,+1/2)$ and $|h| > 0$ small enough,
\begin{equation}\label{EstimateUniform}\begin{split}
|\varphi_k(x+h)-\varphi_k(x)|&=\left|\int_{x}^{x+h}\varphi^{\prime}_k(x)\ \ud x \right| \\
&\leq \sqrt{|h|} \cdot \|\varphi^{\prime}_k\|_{L^2((-1/2,+1/2))} \\
&\leq \sqrt{|h|} \cdot \tilde{c}
\end{split}
\end{equation}
for some constant $\tilde{c} > 0$ independent of $k$. Hence, from \eqref{EstimateUniform} we conclude that $(\varphi_k)$ are uniformly continuous. Also, since $(\varphi_k)$ forms a bounded sequence in $H^1((-1/2,+1/2))$ (note that $\|\varphi_k\|_{L^2((-1/2,+1/2))}=1$), there exists a subsequence, which we also shall denote by $(\varphi_k)$, that converges weakly in $H^1((-1/2,+1/2))$ and in norm in $L^2((-1/2,+1/2))$ to a limit function $\varphi_{\infty}$. Furthermore, on an interval $(\varepsilon,+1/2)$ (or similarly on $(-1/2,\varepsilon)$), the eigenvalue equation implies
\begin{equation*}
\int_{\varepsilon}^{+1/2}|\varphi^{\prime \prime}_k(x)|^2\ \mathrm{d}x \leq \tilde{C}
\end{equation*}
for some uniform constant $\tilde{C} > 0$ and all $k$ large enough. This implies that the sequence $(\varphi^{\prime}_k)$ is also a bounded sequence in $H^1((\varepsilon,+1/2))$ and therefore contains a subsequence that converges in norm in $L^2((\varepsilon,+1/2))$. Hence, referring to a sequence $(\varepsilon=:\varepsilon_n:=1/n)$ and employing Cantor's diagonalisation argument, we find a sequence of ground states for which both, $(\varphi_k)$ and $(\varphi^{\prime}_k)$, converge in norm in $L^2(+\varepsilon_n,+1/2)$ for each $n \in \mathbb{N}$. From the definiton of the weak derivative, we hence conclude that $\varphi_{\infty}$ is in $H^1((\varepsilon_n,+1/2))$ for each $n \in \mathbb{N}$. Also, $(\varphi_k)$ converges in norm to $\varphi_{\infty}$ in $H^1((\varepsilon_n,+1/2))$ for each $n \in \mathbb{N}$. In addition, since 
\begin{equation*}
\|\varphi_\infty\|_{H^1((\varepsilon_n,+1/2))} \leq c
\end{equation*}
with a constant $c > 0$ independent of $n \in \mathbb{N}$, we infer that indeed $\varphi_{\infty} \in  H^1((-1/2,0)) \oplus H^1((0,+1/2))$.

We now aim to prove $\varphi_{\infty}(0^-)=0$ (in the same way one obtains $\varphi_{\infty}(0^+)=0$): We have, using Cauchy-Schwarz, 
\begin{equation*}\begin{split}
|\varphi_{\infty}(0^-)-\varphi_k(0)|&\leq \int_{-1/2}^{0}|\varphi^{\prime}_{\infty}(x)-\varphi^{\prime}_k(x)|\ \ud x\ , \\
&\leq \sqrt{\varepsilon_n}\cdot \|\varphi^{\prime}_{\infty}-\varphi^{\prime}_k\|_{L^2(-\varepsilon_n,0)}+\|\varphi^{\prime}_{\infty}-\varphi^{\prime}_k\|_{L^2(-1/2,-\varepsilon_n)}\ , \\
&\leq \varepsilon\ ,
\end{split}
\end{equation*}
for every $\varepsilon > 0$ and $k,n$ large enough. On the other hand, we readily conclude that $\lim_{k \rightarrow \infty} \varphi_k(0)=0$, since otherwise with \eqref{EstimateUniform} we could find a neighbourhood $U$ of $x=0$ on which $|\varphi_k(x)| \geq \alpha > 0$ uniformly in $k$ and hence
\begin{equation}\label{EstX}\begin{split}
\lambda_0(L_k) \geq \int_{U} w_{L_k}(x)|\varphi_k(x)|^2\ \ud x&=L^2_k\int_{U} v(L_k x)|\varphi_k(x)|^2\ \ud x \\
& \geq \alpha^2 L^2_k \int_{U} v(L_k x)\ \ud x \\
& \geq \alpha^2 L_k \int_{(-b,+b)} v(x)\ \ud x \rightarrow \infty
\end{split}
\end{equation}
for some fixed  $b > 0$ and $k$ large enough (here $v \not\equiv 0$ is used). This is a contradiction and therefore $\lim_{k \rightarrow \infty} \varphi_k(0)=0$ which implies $\varphi_{\infty}(0^-)=0$. In total, we conclude that  $\varphi_{\infty} \in H^1_0((-1/2,+1/2))$.

It remains to prove norm convergence of $(\varphi_k)$ to $\varphi_{\infty}$ in $H^1((-1/2,+1/2))$. We start by observing that, for every $\varepsilon > 0$, we can find a $n \in \mathbb{N}$ such that 
\begin{equation*}
\int_{-\varepsilon_n}^{+\varepsilon_n}|\varphi^{\prime}_k(x)|^2\ \mathrm{d}x \leq \varepsilon
\end{equation*} 
holds for all $k$ large enough. This follows from the variational principle (using that $\lambda_0(L_k) < \frac{\pi^2}{(1/2)^2}$ for all $k$ large enough) in combination with $\varphi_{\infty} \in H^1_0((-1/2,0)) \oplus H^1_0((0,+1/2))$ and the convergence of $(\varphi_k)$ to $\varphi_{\infty}$ in $H^1((\varepsilon_n,+1/2))$. From this we conclude that, for every $\varepsilon > 0$ and $k$ large enough,  
\begin{equation*}\begin{split}
\int_{-1/2}^{+1/2}|\varphi^{\prime}_{\infty}(x)|^2\ \mathrm{d}x - \varepsilon\leq \int_{-1/2}^{+1/2}|\varphi^{\prime}_k(x)|^2\ \mathrm{d}x \leq \int_{-1/2}^{+1/2}|\varphi^{\prime}_{\infty}(x)|^2\ \mathrm{d}x +\varepsilon \ ,
\end{split}
\end{equation*}
and this proves convergence in norm, taking weak convergence into account.
\end{proof}
%
%
	%
	%
	%
%
%
Theorem~\ref{ThmLowerBoundI} shows that the spectral gap closes faster than $L^{-2}$ given the potential decays faster than $|x|^{-2}$. This raises the question as to whether this result can be extended to potentials that decay exactly like $|x|^{-2}$. 
The next theorem provides a negative answer.
\begin{theorem}\label{TheoremSP} Consider the potential $v \in L^{\infty}(\mathbb{R})$ defined via
\begin{equation*}
v(x):=\frac{\eins_{x \geq 1}}{x^2} \ ,
\end{equation*}
where $\eins_{x \geq 1}$ denotes the characteristic function of $(1,\infty)$. Then there exist constants $0 < \alpha < \beta$ such that
\begin{equation*}
\frac{\alpha}{L^2} \leq \Gamma_v(L) \leq \frac{\beta}{L^2}
\end{equation*}
for all $L >0$ large enough.
\end{theorem}
\begin{proof} The upper bound directly follows from Theorem~\ref{UpperBoundI}. 
	
	Furthermore, by Prop.~\ref{UnitaryTrafo} it is enough to prove that there exists $\alpha > 0$ with 
		\begin{equation}\label{ToShow}
		\alpha < \widetilde{\Gamma}_{w_L}(L)
		\end{equation}
		for all $L$ large enough. We note that $w_L(x)=L^2v(Lx)=\frac{\eins_{Lx \geq 1}}{x^2}$. Now, if \eqref{ToShow} does not hold, there exists a sequence of lengths $(L_k)$ such that $\widetilde{\Gamma}_{w_{L_k}}(L_k) \rightarrow 0$ as $k \rightarrow \infty$. 
		
		We now assume that such a sequence exists and consider the spectral gap $\widetilde{\Gamma}_{w_{\infty}}$ of the operator
		\begin{equation*}\begin{split}
		h_{\infty}&=-\frac{\ud^2}{\ud x^2}+\frac{\eins_{x \geq 0}}{x^2} \\
		&:=-\frac{\ud^2}{\ud x^2}+w_{\infty}(x)
		\end{split}
		\end{equation*}
		defined on $L^2((-1/2,+1/2))$. We note that this operator can be defined rigorously via the quadratic form 
		\begin{equation*}
		q_{\infty}[\varphi]=\int_{(-1/2,+1/2)}\left(|\varphi^{\prime}|^2+w_{\infty}(x)|\varphi|^2\right)\ \ud x  
		\end{equation*}
		with form domain
		\begin{equation*}
		\cD_{q_{\infty}}:=\{\varphi \in H^1_0((-1/2,+1/2)):\ \|\sqrt{w_{\infty}}\varphi\|_{L^2((-1/2,+1/2))} < \infty \}\ .
		\end{equation*}
		Since $\cD_{q_{\infty}}$ is, with respect to the form norm, compactly embedded in $L^2((-1/2,+1/2))$, $h_{\infty}$ has purely discrete spectrum. Note that, by continuity, each funtion in $\cD_{q_{\infty}}$ satisfies Dirichlet boundary conditions also at $x=0$.
		
		We now want to show that $\widetilde{\Gamma}_{w_{\infty}} > 0$, i.e., that the ground state of $h_{\infty}$ is non-degenerate: Assume that $\psi_0$ is a ground state to $h_\infty$ associated with the eigenvalue $\mu_0$. Then, by the variational principle,
		\begin{equation*}\begin{split}
		\mu_0&=\frac{\|\psi^{\prime}_0\|^2_{L^2((-1/2,+1/2))}+\|\sqrt{w_{\infty}}\psi_0\|^2_{L^2((-1/2,+1/2))}}{\|\psi_0\|^2_{L^2((-1/2,+1/2))}} \\
		&\geq \frac{\|\psi^{\prime}_0\eins_{x \leq 0}\|^2_{L^2((-1/2,+1/2))}}{\|\psi_0\eins_{x \leq 0}\|^2_{L^2((-1/2,+1/2))}}\ ,
		\end{split}
		\end{equation*}
		where we employed the inequality $\frac{a+b}{c+d} \geq \min\{\frac{a}{c},\frac{b}{d}\}$ for $a,b,c,d > 0$. Hence, since the ground state minimizes the Rayleigh quotient, we conclude that the restriction of $\psi_0$ to the interval $(-1/2,0)$ is also minimizes the Rayleigh quotient. Consequently, all ground states of $h_{\infty}$ agree on $(-1/2,0)$ but this leads to a contradiction when assuming the ground state has higher multiplicity, taking orthogonality into account.
		
		Now, we go back to the sequence of lengths $(L_k)$ such that $\widetilde{\Gamma}_{w_{L_k}}(L_k) \rightarrow 0$ as $k \rightarrow \infty$: As in the proof of Theorem~\ref{ThmLowerBoundI}, we conclude that the sequence of normalized ground states $(\varphi^{L_k}_0)$ of $g_{L_k}$ contains a subsequence that converges in $H^1((-1/2,+1/2))$ to a limit function $\varphi_{\infty} \in H^1((-1/2,+1/2))$ with $\varphi_{\infty}(0)=0$. Furthermore, by Fatou's Lemma, we conclude that $\varphi_{\infty} \in \cD_{q_\infty}$. Now, the variational principle in combination with an operator bracketing argument implies that (along a subsequence)
		\begin{equation*}
		\mu_0 \leq \lim_{k \rightarrow \infty}\lambda_0(L_k) \leq \mu_0\ 
		\end{equation*}
and hence $\varphi_{\infty}$ is the ground state of $h_{\infty}$. 

In the same way, a subsequence of normalized eigenstates $(\varphi^{L_k}_1)$ associated with the eigenvalues $\lambda_1(L_k)$ converges in $H^1((-1/2,+1/2))$ to a limit function $\theta_{\infty}$ which is then orthogonal to $\varphi_{\infty}$. Hence, along a subsequence and by the variational principle and an operator-bracketing argument, 
\begin{equation*}
\mu_1 \leq \lim_{k \rightarrow \infty} \lambda_1(L_k) \leq \mu_1
\end{equation*}
	where $\mu_1$ denotes the second eigenvalue of $h_{\infty}$. This proves the statement since $\mu_1 > \mu_0$, contradicting $\widetilde{\Gamma}_{w_{L_k}}(L_k) \rightarrow 0$ along a suitable subsequence.

\end{proof}
\begin{rem} Note that Theorem~\ref{TheoremSP} shows that the upper bound specified in Theorem~\ref{UpperBoundI} for short-range potentials is indeed optimal within this class of potentials. 
\end{rem}
In our final result, we provide an example where the gap closes like $L^{-3}$.
\begin{prop}\label{Example} Assume that $v \in L^{\infty}(\mathbb{R})$ is the step-potential given by
	\begin{equation*}
	v(x):=v_0 \cdot \eins_{-b\leq x \leq +b}
	\end{equation*}
	with some constants $v_0,b > 0$. Then 
	\begin{equation*}
	\frac{\alpha}{L^3} \leq \Gamma_v(L) \leq \frac{\beta}{L^3}
	\end{equation*}
	for some constant $\beta,\alpha > 0$ and all $L > 0$ large enough.
\end{prop}
\begin{proof} Again we employ the unitary transformation from Proposition~\ref{UnitaryTrafo} and set $w_L(x):=v_0L^2\eins_{-\frac{b}{L}\leq x \leq +\frac{b}{L}}(x)$. Accordingly, we have to show that
	\begin{equation*}
	\frac{\alpha}{L} \leq \widetilde{\Gamma}_{w_L}(L) \leq \frac{\beta}{L}
	\end{equation*}
	for some constants $\beta,\alpha > 0$ and all $L > 0$ large enough. In a first step we employ symmetry of the potential to obtain $\left(\tilde{\varphi}^L_0\right)^{\prime}(0)=0$ for the ground state of $g_L$ and $\tilde{\varphi}^L_1(0)=0$ for the second eigenfunction. Furthermore, for $L > 0$ large enough, we can identifiy the interval $(-1/2,-b/L)$ with $I_1:=(0,l_1)$ and where $l_1:=1/2-b/L$; in the same way, we identify $(-b/L,0)$ with $I_2:=(0,l_2)$ where $l_2:=b/L$. 
	
	Hence, for each eigenfunction $\tilde{\varphi}^L_j$ (when considered on the intervals $I_1,I_2$), $j=0,1$, and $L > 0$ large enough we can make the ansatz
	\begin{equation*}
	\tilde{\varphi}^L_j(x):=\sin(\omega_j x)\ , \quad x \in I_1\ ,
	\end{equation*}
	and 
		\begin{equation*}
	\tilde{\varphi}^L_j(x):=a_j\sinh(M_j x)+b_j\cosh(M_j x)\ , \quad x \in I_2\ ,
	\end{equation*}
	setting $\omega_j=\sqrt{\lambda_j(L)}$ and $M_j:=\sqrt{v_0L^2-\omega^2_j}$. Now, since eigenfunctions of $g_L$ are continuously differentiable, we obtain the conditions
	\begin{equation*}
	a_j:=\frac{\omega_j}{M_j}\cos(\omega_j l_1)\ , \quad b_j=\sin(\omega_j l_1)\ .
	\end{equation*}
	Furthermore, by the conditions mentioned further above, 
	\begin{equation*}
	\omega_0 \cos(\omega_0 l_1)\cosh(M_0 l_2)+ M_0\sin(\omega_0 l_1) \sinh(\omega_0 l_2)=0\ ,
	\end{equation*}
	and 
		\begin{equation*}
	\omega_1 \cos(\omega_1 l_1)\cosh(M_1 l_2)+ M_1\sin(\omega_1 l_1) \sinh(\omega_1 l_2)=0\ .
	\end{equation*}
	Consequently, we obtain 
	\begin{equation*}
	\tan(\omega_0 l_1)=-\frac{\omega_0}{M_0}\left(\tanh(M_0 l_2)\right)^{-1}\ ,
	\end{equation*}
	as well as
	\begin{equation*}
	\tan(\omega_1 l_1)=-\frac{\omega_1}{M_1}\left(\tanh(M_1 l_2)\right)\ .
	\end{equation*}
	This yields 
	\begin{equation*} \begin{split}
	\omega_1-\omega_0&=l^{-1}_1\left[\arctan\left(\frac{\omega_0}{M_0}\left(\tanh(M_0 l_2)\right)^{-1}\right)-\arctan\left(\frac{\omega_1}{M_1}\left(\tanh(M_1 l_2)\right)\right)  \right] \\
	&=l^{-1}_1\int_{X_1}^{X_0}\frac{1}{1+t^2}\ud t
	\end{split}
	\end{equation*}
	with $X_0:=\frac{\omega_0}{M_0}\left(\tanh(M_0 l_2)\right)^{-1}$ and $X_1:=\frac{\omega_1}{M_1}\tanh(M_1 l_2)$. This yields
	\begin{equation*}
\frac{l^{-1}_1(X_0-X_1)}{1+X_0^2}\leq	\omega_1-\omega_0 \leq \frac{l^{-1}_1(X_0-X_1)}{1+X_1^2}\ .
	\end{equation*}
	On the other hand, since $M_jl_2 \rightarrow 1$ as $L \rightarrow \infty$ we obtain, for $L > 0$ large enough,
\begin{equation*}
\frac{c_1}{L}\leq	\omega_1-\omega_0 \leq \frac{c_2}{L}\ .
\end{equation*}
for some constants $c_1,c_2 > 0$ independent of $L$. Consequently, for $L > 0$ large enough,
\begin{equation}\label{EQXXXX}
\frac{\alpha}{L}\leq \lambda_1(L)-\lambda_0(L)=(\omega_1-\omega_0)(\omega_1+\omega_0) \leq \frac{\beta}{L}
\end{equation}
for some constants $\alpha,\beta > 0$ independent of $L > 0$. This proves the statement. 
	\end{proof}

\begin{rem} In Proposition~\ref{Example} we derived upper and lower bounds on the spectral gap for the symmetric step potential that asymptotically behave like $L^{-3}$. We suspect that the same order should hold for any bounded, non-negative and non-zero potential of compact support; possibly, non-negativity can even be relaxed to strictly positive average. Indeed, after the unitary transformation of Prop.~\ref{UnitaryTrafo}, the potential is $w_L(x):=L^2v(Lx)=L \cdot (Lv(Lx))$ and one has
 \begin{equation*}
 Lv(Lx) \rightarrow \left( \int_{\mathbb{R}} v(y) \mathrm{d} y\right) \delta(x)\ , \quad L \rightarrow \infty\ ,
 \end{equation*}
 in the sense of distributions. 
 But a straightforward calculation shows that the spectral gap of the Dirichlet Laplacian on $(-1/2,+1/2)$ with a $\delta$-interaction of strength $L$ at $x = 0$ behaves approximately as $L^{-1}$ as $L$ tends to infinity, suggesting a similar asymptotics as in Proposition~\ref{Example}.	%

\end{rem}

\appendix

\section{A shortcut in the proof of Theorem~\ref{ThmLowerBoundI}}

\begin{lemma}
	\label{lem:appendix}
	 Assume that $v \in L^{\infty}(\mathbb{R})$ is non-negative and not the zero potential. In addition, we assume it decays such that
	\begin{equation*}
	|v(x)| \leq \frac{C}{|x|^{\alpha}}\ , \quad \text{for a.e.}\quad x\in \mathbb{R}\ ,
	\end{equation*}
	for some constants $C > 0$ and $\alpha > 2$. Then, for the scaled potentials $w_L := L^2 v(L \cdot) \in L^\infty((- \frac{1}{2}, +\frac{1}{2} ))$ and the ground state eigenvalue $\lambda_0(L)$ of the operator 
	\[
	g_L = -\frac{\mathrm{d}^2}{\mathrm{d} x^2} + w_L(x)
	\quad
	\text{in $L^2((- 1/2, +1/2 ))$}
	\]
	with Dirichlet boundary conditions, one has
	\begin{equation}
	\label{eq:liminf}
		\liminf_{L \to \infty}
		\lambda_0(L)
		\geq
		\frac{\pi^2}{(1/2)^2}.
	\end{equation}
\end{lemma}

\begin{proof}
It is easy to see that the ground state eigenvalue $\lambda_0(L)$ remains bounded as $L \to \infty$.
Let us prove that if a function $\phi \in H^1_0((-1/2,+1/2))$ with $\lVert \phi \rVert_{L^2((-1/2,+1/2))} = 1$ satisfies 
\begin{equation}
	\label{eq:condition_epsilon}
	\min_{\lvert x \rvert \leq \epsilon} \lvert \phi(x) \rvert \leq \epsilon
\end{equation}
for sufficiently small $\epsilon > 0$, then 
\begin{equation}
	\label{eq:lower_bound_appendix}
	\lVert \phi' \rVert_{L^2((-1/2, +1/2))}^2
	\geq
	\frac{\pi^2}{(1/2 + \epsilon)^2}
	\lVert \phi \rVert_{L^2((-1/2, +1/2))}^2
	-
	\tilde C
	\epsilon
\end{equation}
for an $\phi$-independent $\tilde C > 0$.
Indeed, given such a $\phi$, we can add a function $\eta \in H^1_0((-1/2,+1/2))$ with $\lVert \eta \rVert_{H^1((-1/2,+1/2))} \leq C \epsilon$ such that $\phi - \eta$ has a zero $x_0 \in (- \epsilon, \epsilon)$.
The difference $\phi - \eta$ can be considered as a function in $H^1_0((-1/2,x_0)) \oplus H^1_0((x_0, 1/2))$ and, using $\max \{ 1/2 \pm x_0 \} \leq 1/2 + \epsilon$, the Poincar\'e inequality with optimal constant, applied on each subinterval separately, yields
\[
	\lVert \phi' - \eta' \rVert_{L^2((-1/2,+1/2))}^2
	\geq
	\frac{\pi^2}{(1/2 + \epsilon)^2}
	\lVert \phi - \eta \rVert_{L^2((-1/2,+1/2))}^2\ .
\]
This implies, for sufficiently small $\epsilon > 0$,
\begin{align*}
	\lVert \phi' \rVert_{L^2((-1/2,+1/2))}
	&\geq
	\lVert \phi' - \eta' \rVert_{L^2((-1/2,+1/2))} - \lVert \eta' \rVert_{L^2((-1/2,+1/2))}
	\\
	&\geq
	\frac{\pi}{1/2 + \epsilon}
	\lVert \phi  \rVert_{L^2((-1/2,+1/2))}
	-
	C \epsilon
	\geq 0
\end{align*}
which leads to~\eqref{eq:lower_bound_appendix}. By letting $\epsilon \to 0$, the estimate \ref{eq:liminf} follows.

It therefore remains to check~\eqref{eq:condition_epsilon} for a sequence of ground states $(\varphi_k)$ of the operator $g_{L_k}$ with associated lengths $(L_k)_{k \in \mathbb{N}}$, $L_k \to \infty$.
We claim that for every $\epsilon > 0$, there exists a subsequence $(L_{k_j})_{j \in \mathbb{N}}$ such that $\varphi_{k_j}$ satisfies~\eqref{eq:condition_epsilon}.
Indeed, if this was not the case, then for all $k$ large enough, we would have $\min_{\lvert x \rvert \leq \epsilon} \lvert \varphi_{k}(x) \rvert \geq \epsilon$ and consequently
\[
	\lambda_0(L_{k})
	\geq
	\int_{- \epsilon}^\epsilon
	w_{L_{k}}(x) \lvert \varphi_{k}(x) \rvert ^2\
	\mathrm{d} x
	\geq
	\epsilon^2
	L_{k}^2 
	\int_{- \epsilon}^\epsilon
	v(L_{k} x)\
	\mathrm{d} x
	\to
	\infty \ ,
\]
a contradiction to the boundedness of $\lambda_0(L_{k})$.
\end{proof}

\vspace*{0.5cm}

\subsection*{Acknowledgement}{}JK would like to thank K.~Pankrashkin for stimulating discussions surrounding the topic of this paper. We also thank M.~Plümer and A.~Seelmann for reading the manuscript and useful remarks.

\vspace*{0.5cm}

{\small
\bibliographystyle{amsalpha}
\bibliography{Literature}}

\def\cprime{$'$} \def\polhk#1{\setbox0=\hbox{#1}{\ooalign{\hidewidth
  \lower1.5ex\hbox{`}\hidewidth\crcr\unhbox0}}}
\providecommand{\bysame}{\leavevmode\hbox to3em{\hrulefill}\thinspace}
\providecommand{\MR}{\relax\ifhmode\unskip\space\fi MR }
\providecommand{\MRhref}[2]{%
  \href{http://www.ams.org/mathscinet-getitem?mr=#1}{#2}
}
\providecommand{\href}[2]{#2}
\begin{thebibliography}{KPS20}

\bibitem[AB89]{AB}
M.~S. Ashbaugh and R.~Benguria, \emph{Optimal lower bound for the gap between
  the first two eigenvalues of one-dimensional {S}chr\"{o}dinger operators with
  symmetric single-well potentials}, Proc. Amer. Math. Soc. \textbf{105}
  (1989), no.~2, 419--424.

\bibitem[Abr91]{Abramovich}
S.~Abramovich, \emph{The gap between the first two eigenvalues of a
  one-dimensional {S}chr\"{o}dinger operator with symmetric potential}, Proc.
  Amer. Math. Soc. \textbf{111} (1991), no.~2, 451--453.

\bibitem[ACH]{ACH}
B.~Andrews, J.~Clutterbuck, and D.~Hauer, \emph{{The fundamental gap for a
  one-dimensional Schrödinger operator with Robin boundary conditions}},
  preprint, arXiv:2002.06900v1.

\bibitem[KPS20]{KPSBEC}
J.~Kerner, M.~Pechmann, and W.~Spitzer, \emph{On a condition for type-{I}
  {B}ose-{E}instein condensation in random potentials in {$d$} dimensions}, J.
  Math. Pures Appl. (9) \textbf{143} (2020), 287--310.

\bibitem[KS86]{KirschGapII}
W.~Kirsch and B.~Simon, \emph{Lifshitz tails for periodic plus random
  potentials}, J. Statist. Phys. \textbf{42} (1986), no.~5-6, 799--808.

\bibitem[KS87]{KirschGap}
\bysame, \emph{Comparison theorems for the gap of {S}chr\"{o}dinger operators},
  J. Funct. Anal. \textbf{75} (1987), no.~2, 396--410.

\bibitem[Lav94]{Lavine}
R.~Lavine, \emph{The eigenvalue gap for one-dimensional convex potentials},
  Proc. Amer. Math. Soc. \textbf{121} (1994), no.~3, 815--821.

\bibitem[LL01]{lieb2001analysis}
E.~H. Lieb and M.~Loss, \emph{{Analysis}}, American Mathematical Society,
  Providence, R.I., 2001.

\end{thebibliography}

\end{document}